\documentclass[jap,press]{apt} \volno{63} \partno{1}
\renewcommand{\crnotice}{}
\paperno{}
\firstpage{1}
\usepackage[colorlinks = true,linkcolor = blue,urlcolor  = blue,citecolor = blue,anchorcolor = blue]{hyperref}

\usepackage{amsmath,amssymb,amsfonts}
\usepackage{bm}
\usepackage{graphicx}
\usepackage{xcolor}
\usepackage{hyperref}
\usepackage{enumitem}
\usepackage{mathrsfs}
\usepackage{accents}
\usepackage{mathtools}
\usepackage{authblk}
\usepackage{mathrsfs}
\usepackage{booktabs}
\usepackage{float}
\usepackage{dcolumn}
\usepackage{comment}

\newcommand{\p}{\kP}
\newcommand{\E}{\kE}

\newcommand{\dee}{\sd}

\newcommand{\ind}{\mathbf{1}}

\newcommand{\R}{\mathbb{R}}

\newcommand{\as}{\text{a.s.}}

\newcommand{\eqd}{\eqdist}

\newcommand{\level}{\delta}
\newcommand{\gridx}{\eta}

\usepackage[colorlinks = true,linkcolor = blue,urlcolor  = blue,citecolor = blue,anchorcolor = blue]{hyperref}
\authornames{M.~FISCHER AND P.~W.~GLYNN}
\shorttitle{Conventional Confidence Interval Coverage Proportions}
\hypersetup{pdftitle={On the Proportion of Time that a Conventional Confidence Interval Covers},
            pdfauthor={Monte Fischer and Peter W. Glynn}, pdflang={en}, pdfdisplaydoctitle=true}

\begin{document}%

\title{On the Proportion of Time that a Conventional Confidence Interval Covers}
\date{November 2025}

\noindent\authoronenoaffil{Monte Fischer}
\authortwo[Stanford University]{Peter W.~Glynn} 

\emailone{monte@montefischer.com}
\emailtwo{glynn@stanford.edu}

\begin{abstract}
    The concept of confidence sequences, or `anytime confidence intervals,' has attracted significant interest in recent years. This motivates the question of how well conventional confidence intervals behave when similarly evaluated. In particular, we study the probability distribution of the proportion of time a conventional confidence interval procedure covers the population mean, when the total time $n$ spent sampling is large. We establish an invariance principle for the limiting distribution as $n$ tends to infinity and show that the limit random variable $Z(1)$ can be expressed as a functional of Brownian motion. The random variable $Z(1)$ exhibits a scale invariance property that enables the derivation of a partial differential equation (PDE) for the density of $Z(1)$ jointly with the Brownian motion at $t=1$, namely $B(1)$. We show that this density can also be computed as the stationary distribution of a two-dimensional singular diffusion, one component of which is an Ornstein-Uhlenbeck process. We then derive an approximating Markov chain that allows us to numerically solve the PDE. This density has a number of interesting features that provide various insights into the time-averaged coverage proportion problem. 
\end{abstract}

\keywords{Coverage probability; anytime-valid inference; sequential analysis; Brownian motion; Ornstein-Uhlenbeck process; occupation time; scale invariance; time-average coverage; large-sample theory; numerical discretization; Monte Carlo simulation}%

\ams{60F17}{62E20}

\section{Introduction}
Suppose that we sample independent and identically distributed (i.i.d.) real-valued observations $X_1, X_2, \dots$ from an unknown population, with the goal of using the sample to estimate $\E X$, where $X$ is a generic observation from the population. Given a fixed number of observations $X_1, X_2, \dots X_k$, it is standard statistical practice to construct an accompanying approximate confidence interval $\mathscr I_k$ for $\E X$ based on a procedure which guarantees asymptotic coverage, such that $\liminf_{k \to \infty} \p(\E X \in \mathscr I_k) \ge 1-\delta$ for a chosen $\delta > 0$. In recent years, there has been increasing interest in the development and application of \textit{confidence sequences}, or `anytime confidence intervals,' in which the goal is to construct a sequence $(I_k : k \ge 1)$ of intervals such that the time-uniform coverage probability $\p (\E X \in I_k , k \ge 1) \ge 1-\level$, for $\level > 0$; see \cite{howard2021time}. In order that the coverage be uniform in $k$, these intervals must, on average, be larger than conventional intervals $\mathscr I_k$. Given the ubiquitous use of conventional confidence intervals, it is of interest to understand the characteristics of the proportion of time that the sequence $\mathscr I_1$, $\mathscr I_2$, \dots, $\mathscr I_n$ covers $\E X$ when $n$ is large. 

In particular, the standard $100(1-\level)\%$ large-sample approximate confidence interval $\mathscr I_k$ constructed from $X_1, X_2, \dots, X_k$ is given by 
\[ \mathscr I_k = [\bar X_k - z \frac{s_k}{\sqrt k} , \bar X_k + z \frac{s_k}{\sqrt k}] \]
for $k \ge 2$, 
where $\bar X_k = (X_1 + \cdots + X_k) / k$ is the sample mean, $s_k = (\sum_{i=1}^k (X_i - \bar X_k)^2 / (k-1))^{1/2}$ is the sample standard deviation, and $z$ is chosen so that $\p(-z \le \mathcal N(0,1) \le z) = 1-\level$, 
where $\mathcal N(0,1)$ is a normal random variable (RV) with mean 0 and variance 1. 
(For $k=1$, $\mathscr I_1$ can be defined arbitrarily.) 
Our goal in this work is to study the RV $\mathscr P_n$ defined by 
\[ \mathscr P_n = \frac{1}{n} \sum_{k=1}^n \ind(\E X \in \mathscr I_k), \]
namely the proportion of the intervals $\mathscr I_1, \mathscr I_2, \dots, \mathscr I_n$ that cover $\E X$.

In Section \ref{sec:2}, we establish an invariance principle for $\mathscr P_n$ (Theorem 1), proving that when $0 < \E X^2 < \infty$, 
\begin{equation} \label{1.1}
    \mathscr P_n \Rightarrow Z(1)
\end{equation}
as $n \to \infty$, where 
\[ Z(t) = \int_0^t \ind(-z \le \frac{B(s)}{\sqrt s} \le z) \dee s \]
for $t \ge 0$, $B = (B(s) : s \ge 0)$ is a standard Brownian motion (with drift 0 and unit volatility), and $\Rightarrow$ denotes weak convergence. The relation \eqref{1.1} establishes that the limit RV $Z(1)$ is invariant to the fine structure of the RV $X$, and is valid so long as $X$ has finite and positive variance. Note that despite the fact that $\mathscr P_n$ is an average, \eqref{1.1} shows that it does not converge to its asymptotic expectation $1 - \level$, and hence it does not obey a conventional law of large numbers. 

The limit theorem of Section \ref{sec:2} is formulated in such a way that it covers virtually any estimator that satisfies a central limit theorem (CLT) for which there is an accompanying strongly consistent variance estimator. Hence, $Z(1)$ characterizes the limiting proportion in much greater generality than just the setting in which a mean $\E X$ is being estimated. For example, the estimation of a smooth function of expectations $g(EX)$ is covered by our theory, where $g$ is continuously differentiable and $X$ is (potentially) vector-valued, as is quantile estimation, steady-state estimation, and many other problems.

In Section \ref{sec:3}, we develop several partial differential equations (PDEs) that permit computation of various expectations associated with $Z(1)$. One PDE (see \eqref{pde}) leverages an important scaling property of the RV's $(Z(t), B(t))$ (as a function of $t$), and permits computation of the joint density of the RV $(Z(1), B(1))$. We then prove that this PDE can also be viewed as the PDE associated with the stationary density of a certain Markov process.

Section \ref{sec:4} then studies the density of $Z(1)$ numerically, by solving the PDE. In particular, we develop a carefully constructed Markov chain approximation to the Markov process and use its stationary distribution as an approximation to the density. It turns out that the density of $Z(1)$ exhibits a surprisingly wide variety of different shapes as a function of the parameter $z$ (or, equivalently, $\level$). One important feature of the density is that when the coverage $1 - \level$ is close to 1, $\mathscr P_n$ is greater than $1 - \level$ with surprisingly high probability. So, the proportion of time the confidence intervals cover $\E X$ tends (most of the time) to be significantly higher than $1-\level$. This is balanced by the fact that there are low probability paths where the confidence intervals do not cover $\E X$ for a substantial proportion of time. The conditional densities also exhibit shapes that provide insight into the coverage proportion problem. Section \ref{sec multivariate} concludes the paper with a brief discussion of how the theory generalizes to the setting of $d$-dimensional confidence region coverage.

We note that this theory is directly relevant to much of the statistical and simulation literature that uses Monte Carlo experimentation to evaluate the coverage of newly proposed statistical and simulation procedures for producing confidence intervals. It is common practice to use Monte Carlo to compute the coverage probability for the proposed interval at different statistical sample sizes or simulation run lengths; see for example \cite{choquet1999bootstrap,flegal2010batch,owen2025better}. When estimating such coverages at different sample sizes $k_1 < k_2 < \cdots < k_r$ it is natural to simply run repeated simulations to sample size $k_r$, and to then check the coverages at sample sizes $k_1, k_2, \dots, k_{r-1}$ along the same sample path. When computing coverage probabilities via the use of common sample paths, the RV $\mathscr P_n$ provides a description of the coverage fraction along the entire path when $k_i = i$, $1 \le i \le n$. This strong correlation of the coverage fraction along such sample paths suggests that a better practice, in coverage evaluation of statistical and simulation procedures, is to run independent Monte Carlo experiments at each sample size $k_1$, $k_2$, \dots, $k_r$.

The variable $\mathscr P_n$ also has a natural hypothesis testing interpretation. Suppose a null hypothesis is being tested sequentially using a stream of i.i.d.~samples $(X_k : k \ge 1)$. For concreteness, we posit the null hypothesis to be $\E X = 0$. At each time $k$, a conventional confidence interval $\mathscr I_k$ is formed for $\E X$. The statistician provisionally rejects the null at each time $k$ according to whether $0 \notin \mathscr I_k$. Under the null hypothesis, $\mathscr P_n$ is the proportion of times 1 through $n$ for which the statistician (correctly) fails to reject the null hypothesis. The object $\mathscr P_n$ is then of interest, as it describes the proportion of null hypothesis rejections in this repeated hypothesis testing setting. 

The study of Brownian motion in connection with a region with square root boundary was undertaken by \cite{breiman1967first} in the context of the first exit time of Brownian motion from the region. The present study, by comparison, is the first development of the properties of the associated \textit{occupation time} of Brownian motion within the region with square root boundary, over a finite interval. 

\section[An Invariance Principle for Pn]{An Invariance Principle for $\mathscr P_n$} \label{sec:2}

In this section, we consider a general estimation environment in which the goal is to estimate a quantity $\alpha$ via an estimator $\alpha(t)$ based on a sample indexed by the parameter $t$. For example, in the setting discussed in the Introduction, $\alpha = \E X$ and the estimator $\alpha(t)$ is given by $\bar X_{\lfloor t \rfloor }$, where $\lfloor t \rfloor$ is the greatest integer less than or equal to $t$. We assume that $(\alpha(t) : t \ge 0)$ satisfies a functional central limit theorem (FCLT):
\begin{assumption} \label{a1 ci}
    Assume that there exists $\sigma > 0$ for which 
    \[ (\epsilon^{-1/2} (\alpha(t/\epsilon) - \alpha) : 0 < t < \infty) \Rightarrow (\sigma B(t)/t : 0 < t < \infty) \]
    in $D(0,\infty)$ as $\epsilon \downarrow 0$.
\end{assumption}
\noindent Convergence in the space $D(0, \infty)$ means convergence in $D[a,b]$ for each $0 < a < b < \infty$. As such, $x_n \to x_\infty$ in $D(0,\infty)$ when $x_\infty$ is continuous requires that $x_n(\cdot)$ converge to $x_\infty(\cdot)$ uniformly in each compact subinterval of $(0,\infty)$; see, for example, \cite{whitt2002stochastic}.

We further assume the existence of an estimator $s(t)$ for the standard deviation $\sigma$:
\begin{assumption} \label{a2 ci}
    Assume that
    \[ s(t) \to \sigma \quad \as \]
    as $t \to \infty$.
\end{assumption}
In the i.i.d.~setting described in the Introduction, we can take $\alpha(t) = \bar X_{\lfloor t \rfloor}$ and $s(t) = s_{\lfloor t \rfloor}$ for $t \ge 2$, where $(s_n : n \ge 1)$ is the sample standard deviation.

Under Assumptions \ref{a1 ci} and \ref{a2 ci},
\[ t^{1/2} (\alpha(t) - \alpha) / s(t) \Rightarrow \mathcal N(0,1)\]
as $t \to \infty$, so that 
\[ \mathscr I(t) = [\alpha(t) - z \frac{s(t)}{\sqrt t}, \alpha(t) + z \frac{s(t)}{\sqrt t}] \]
is an asymptotic $100(1-\level)\%$ confidence interval for $\alpha$. Then, for $t > 0$, put 
\[ \mathscr P(t) \triangleq \frac{1}{t} \int_0^t \ind(\alpha \in \mathscr I(s)) \dee s, \]
so that $\mathscr P(n)$ generalizes the RV $\mathscr P_n$ described in the Introduction. (As in the Introduction, we may define $\mathscr I(s)$ for $s \le t_0$ arbitrarily, without affecting the limiting behavior of $\mathscr P(t)$.)

As noted in the Introduction, this framework covers a wide variety of different estimation settings. A discussion of the specific estimation environments within which Assumptions \ref{a1 ci} and \ref{a2 ci} hold can be found in, for example, \cite{glynn1992asymptotic}. 

Our main result in this section is the following limit theorem. Recall that we define $Z(t) \triangleq (1/t) \int_0^t \ind(-z \le B(s) / \sqrt s \le z) \dee s$.

\begin{thm}
    Under Assumptions \ref{a1 ci} and \ref{a2 ci},
    \[ \mathscr P(t) \Rightarrow Z(1) \]
    as $t \to \infty$.
\end{thm}

\begin{proof}
    For $x \in D(0, \infty)$, put 
    \[ \varphi(x) = \int_0^1 \ind(-z \le s^{1/2} x(s) \le z) \dee s.\]
    We note that $\varphi$ is continuous on $D(0,\infty)$ at continuous functions $x$ for which $t^{1/2} |x(t)| \ne z$ almost everywhere in $t$. But 
    \[ \E \int_0^1 I\left(\frac{|B(s)|}{\sqrt s} = z\right) \dee s = \int_0^1 \p( |\mathcal N(0,1)| = z) \dee s = 0, \]
    so 
    \[ \int_0^1 I\left(\frac{|B(s)|}{\sqrt s} = z\right) \dee s = 0 \quad \as,\]
    and hence $\varphi$ is continuous at $(B(s) / s : 0 < s < \infty)$ almost surely (a.s.).

    Under Assumptions \ref{a1 ci} and \ref{a2 ci},
    \[ (\epsilon^{-1/2} (\alpha(t/\epsilon) - \alpha) / s(t / \epsilon) : 0 < t < \infty) \Rightarrow (B(t) / t : 0 < t < \infty) \]
    as $\epsilon \downarrow 0$ in $D(0, \infty)$, so 
    \begin{equation} \label{2.1}
        \varphi((\epsilon^{-1/2} (\alpha(t / \epsilon) - \alpha) / s(t / \epsilon) : 0 < t < \infty)) \Rightarrow \varphi((B(t)/t : 0 < t < \infty)) = Z(1)
    \end{equation}
    as $\epsilon \downarrow 0$. But the left-hand side of \eqref{2.1} equals 
    \begin{align*}
        & \int_0^1 \ind(-z \le t^{1/2} \epsilon^{-1/2} ( \alpha(t / \epsilon) - \alpha) / s(t / \epsilon) \le z ) \dee t \\
        &= \epsilon \int_0^{1/\epsilon} I (-z \le u^{1/2} (\alpha(u) - \alpha) / s(u) \le z) \dee u \\
        &= \epsilon \int_0^{1/\epsilon} \ind(\alpha \in \mathscr I(u)) \dee u \\
        &= \mathscr P(1/\epsilon),
    \end{align*}
    proving the theorem.
\end{proof}

We can easily derive a few basic properties of $Z(1)$. First, 
\begin{align}
\E Z(1) &= \int_0^1 \p(-z \le t^{-1/2}B(t) \le z) \dee t \nonumber \\
&= \int_0^1 \p(-z \le \mathcal N(0,1) \le z) \dee t  \nonumber\\
&= 1-\level. \label{2.2}
\end{align}
Also,
\begin{align}
    \E Z(1)^2 &= 2 \int_0^1 \int_s^1 \p\left(-z \le \frac{B(t)}{\sqrt t} \le z, -z \le \frac{B(s)}{\sqrt s} \le z\right) \dee t \dee s \nonumber\\
    &= 2 \int_0^1 \int_{- z \sqrt s}^{z \sqrt s} \int_s^1 \p\left(-z \le \frac{B(t-s) + x \sqrt s}{\sqrt t} \le z\right) \dee t \cdot \p\left(\frac{B(s)}{\sqrt s} \in \dee x\right) \dee s \nonumber\\
    &= 2 \int_0^1 \int_{- z \sqrt s}^{z \sqrt s} \int_s^1 \left[ \Phi\left(\frac{zt^{1/2} - xs^{1/2}}{\sqrt {t-s}}\right) - \Phi\left(\frac{-zt^{1/2} - xs^{1/2}}{\sqrt {t-s}}\right) \right] \dee t \cdot \phi(x) \dee x \dee s, \label{2.3}
\end{align}
where $\phi(\cdot)$ and $\Phi(\cdot)$ are, respectively, the density and cumulative distribution function of a $\mathcal N(0,1)$ RV. Table \ref{tab:1} displays $\sqrt {\var Z(1)}$ for various values of $z = z(\delta)$, computed from \eqref{2.2} and \eqref {2.3} by numerical integration.

\begin{table}[H]
    \centering
\begin{tabular}{rrr}
\toprule
$\level$ & $z(\level)$ & $\sqrt{\var Z(1)}$ \\
\midrule

0.01 & 2.5758 & 0.0537 \\
0.05 & 1.9600 & 0.1280 \\
0.10 & 1.6449 & 0.1813 \\
0.30 & 1.0364 & 0.2757 \\
0.45 & 0.7554 & 0.2844 \\
0.50 & 0.6745 & 0.2787 \\
0.55 & 0.5978 & 0.2690 \\
0.70 & 0.3853 & 0.2166 \\
0.90 & 0.1257 & 0.0904 \\
0.95 & 0.0627 & 0.0478 \\
0.99 & 0.0125 & 0.0102 \\
\bottomrule
\end{tabular}
\caption{Standard deviation of $Z(1)$ associated with various $\level \in (0,1)$.}
\label{tab:1}
\end{table}

\section[PDEs Associated with Z(1)]{PDEs Associated with $Z(1)$} \label{sec:3}
Since $Z(1)$ is derived from the most basic of all diffusion processes, namely standard Brownian motion, it is not surprising that various expectations associated with $Z(1)$ can be computed as solutions to PDEs. 

\begin{prop} ~ \label{prop1}
    \begin{enumerate}[label=\alph*)]
        \item \label{prop1:a} Suppose that $u:[0,1]^2 \times \R \to \R$ is a bounded solution of the PDE 
        \[ \frac{\partial u}{\partial t} (t, x, y) = \frac{\partial u}{\partial x} (t, x, y) \ind(-z \le \frac{y}{\sqrt {1-t}} \le z) + \frac{1}{2} \frac{\partial^2 u}{\partial y^2} (t, x, y),\]
        for $0 \le t, x \le 1$, and $y \in \R$, subject to $u(0,x,y) = h(x)$. Then, $u(1,0,0) = \E h(Z(1))$.
        \item \label{prop1:b} Suppose that $v : \R^2 \to \R$ is a bounded solution of the PDE 
        \[ \frac{\partial v}{\partial t} (t, x) = \theta \ind( -z \le \frac{x}{\sqrt {1-t}} \le z)\,v (t, x) + \frac{1}{2} \, \frac{\partial^2 v}{\partial x^2} (t, x) \]
        for $t \in [0,1]$ and $x \in \R$, subject to $v(0,x) = 1$ for $x \in \R$. Then, 
        \[ \E \exp(\theta Z(1)) = v(1,0). \]
    \end{enumerate}
\end{prop}
Note that the first PDE, in three variables, allows us to compute $\E h(Z(1))$ for a generic (bounded) function $h$. The second PDE, in two variables, permits computation of the moment generating function (MGF) of $Z(1)$ at a fixed $\theta$.

\begin{proof}[Proof of Proposition \ref{prop1}]
    Observe that It\^{o}'s formula implies that 
    \begin{align*}
        &\dee u (1-t, Z(t), B(t)) \\
        &= -\frac{\partial u}{\partial t} u(1-t, Z(t), B(t)) \dee t + \frac{\partial u}{\partial x} (1-t, Z(t), B(t)) \cdot \ind(-z \le \frac{B(t)}{\sqrt{1 - (1-t)}} \le z) \dee t \\
        &\quad + \frac{\partial u}{\partial y} (1 - t, Z(t), B(t)) \dee B(t) + \frac{1}{2} \frac{\partial^2 u}{\partial y^2} (1-t, Z(t), B(t)) \dee t \\
        &= \frac{\partial u}{\partial y} (1-t, Z(t), B(t)) \dee B(t),
    \end{align*}
    where the second equality follows as a consequence of the fact that $u$ satisfies the PDE.

    Because of the assumed boundedness, 
    \[ M(t) \triangleq u(1-t, Z(t), B(t)) \]
    is a martingale adapted to $\mathcal F \triangleq \sigma(B(s): 0 \le s \le t)$ for $t \ge 0$. Hence, 
    \[ \E u(0, Z(1), B(1)) = u(1,0,0).\]
    But $u(0,x,0) = h(x)$, so we conclude that 
    \[ u(1,0,0) = \E h(Z(1)),\]
    yielding part \ref{prop1:a}.

    For part \ref{prop1:b}, we apply It\^{o}'s formula as follows:
    \begin{align*}
        &\dee (\exp (\theta Z(t)) v(1-t, B(t))) \\
        &= \exp(\theta Z(t)) v(1-t, B(t)) \theta \dee Z(t) \\
        &\quad + \exp(\theta Z(t)) \Big[-\frac{\partial v}{\partial t} (1-t, B(t)) \dee t + \frac{\partial v}{\partial x} (1-t, B(t)) \dee B(t) \\
        &\qquad \qquad \qquad \quad+ \frac{1}{2} \frac{\partial^2 v}{\partial x^2} (1-t, B(t)) \dee t \Big] \\
        &= \exp(\theta Z(t)) \Big[\theta \ind(-z \le \frac{B(t)}{\sqrt{1-(1-t)}} \le z) v(1-t, B(t)) \\
        &\qquad \qquad  \qquad \quad - \frac{\partial v}{\partial t}(1-t, B(t)) + \frac{1}{2} \frac{\partial^2 v}{\partial x^2} (1-t, B(t)) \Big] \dee t \\
        &\qquad + \exp(\theta Z(t)) \frac{\partial v}{\partial x} (1-t, B(t)) \dee B(t) \\
        &= \exp(\theta Z(t)) \frac{\partial v}{\partial x} v(1-t, B(t)) \dee B(t).
    \end{align*}
    The PDE satisfied by $v$ was used to obtain the final equality in the above calculation.

    In view of the assumed boundedness of $v$ and the fact that $|Z(t)| \le 1$,
    \[ \widetilde M(t) \triangleq \exp(\theta Z(t)) v(1-t, B(t)) \]
    is a martingale adapted to $(\mathcal F_t : t \ge 0)$. It follows that 
    \[ \E \exp(\theta Z(1)) v(0, B(1)) = v(1,0). \]
    The assumed boundary condition $v(0,x) = 1$ for $x \in \R$ then yields 
    \[ v(1,0) = \E \exp(\theta Z(1)), \]
    as required.
\end{proof}

We derive an additional PDE for the joint density of $(Z(1), B(1))$, taking advantage of the following fundamental scaling relationship.

\begin{prop} \label{prop2}
    For $t \ge 0$,
    \[ (Z(t), B(t)) \overset{\mathcal D}{=} (t Z(1), t^{1/2} B(1)),\] 
    where $\overset{\mathcal D}{=}$ denotes equality in distribution.
\end{prop}

\begin{proof}
    For $t \ge 0$,
    \begin{align*}
    (Z(t), B(t)) &= \left( \int_0^t \ind(-z \le \frac{B(s)}{\sqrt s} \le z) \dee{s}, B(t) \right) \\
    &= \left( \int_0^1 \ind(-z \le \frac{B(tu)}{\sqrt {tu}} \le z)\,t \dee{u}, B(t) \right) \\
    &\overset{\mathcal D}{=} \left( \int_0^1 \ind(-z \le \frac{B(u)}{\sqrt {u}} \le z)\, t \dee{u}, \sqrt t \, B(1) \right) \\
    &= (t Z(1), t^{1/2} B(1)).
\end{align*}
\end{proof}

If $(g_t(x,y): x, y \in \R)$ is the joint density of $(Z(t), B(t))$, Proposition \ref{prop2} implies that 
\begin{equation} \label{eqn:joint_density}
g_t(x,y) = t^{-3/2} g_1(x/t, y/\sqrt t)
\end{equation} 
for $t > 0$, $x, y \in \R$. For $\epsilon \in (0,1)$, 
\begin{align*}
    g_1(x,y) \dee x \dee y &= \int_{\R^2} \p(Z(1) \in \dee x, B(1) \in \dee y \mid Z(1-\epsilon) = u_1, B(1-\epsilon) = u_2) \\
    & \qquad \quad \cdot g_{1-\epsilon}(u_1, u_2) \dee{u_1} \dee{u_2}.
\end{align*}

Since $Z(1) = Z(1-\epsilon) + \epsilon + o(\epsilon)$ when $-z < B(1) < z$, while $Z(1) = Z(1-\epsilon) + o(\epsilon)$ when $|B(1)| > z$, this suggests that 
\begin{align*}
    g_1(x,y) \dee y &= \int_{\R} \p(B(1) \in \dee y \mid B(1-\epsilon) = u_2) \\
    & \qquad \quad \cdot [g_{1-\epsilon}(x-\epsilon, u_2) + o(\epsilon) ] \dee{u_2}
\end{align*}
for $|y| < z$, and 
\begin{align*}
    g_1(x,y) \dee y &= \int_{\R} \p(B(1) \in \dee y \mid B(1-\epsilon) = u_2) \\
    & \qquad \quad \cdot [g_{1-\epsilon}(x, u_2) + o(\epsilon) ] \dee{u_2}
\end{align*}
for $|y| > z$. Because \eqref{eqn:joint_density} holds, and $B(1) - B(1-\epsilon) \overset{\mathcal D}{=} \sqrt \epsilon \mathcal N(0,1)$, it follows that
\begin{align*}
    g_1(x,y) &= \int_{\R} \phi ((y - u_2) / \sqrt \epsilon) \epsilon^{-1/2} \, g_{1-\epsilon} (x - \epsilon \ind(-z \le y \le z), u_2) \dee{u_2} + o(\epsilon) \\
    &= \int_{\R} \phi(w) \, g_{1-\epsilon}(x - \epsilon \ind(-z \le y \le z), y - w \sqrt \epsilon) \dee w + o(\epsilon) \\
    &= \int_{\R} \phi(w) (1 - \epsilon)^{-3/2} g_1 \left( \frac{x - \epsilon I (-z \le y \le z)}{1-\epsilon}, \frac{y - w \sqrt \epsilon}{\sqrt{1-\epsilon}} \right) \dee w + o(\epsilon).
\end{align*}
Furthermore, if $g_1$ is twice continuously differentiable with bounded derivatives, then
\begin{align*}
g_1(x,y) &= \int_{\R} \phi(w) (1 + \tfrac32 \epsilon + o(\epsilon)) \\
&\qquad \quad \cdot g_1((x-\epsilon \ind(-z \le y \le z)) (1 + \epsilon) + o(\epsilon), (y - w \sqrt \epsilon  ) (1 + \tfrac\epsilon2)+ o(\epsilon)) \dee{w} + o(\epsilon) \\
&= \int_\R  \phi(w) (1 + \tfrac32 \epsilon) \Big[ g_1 (x,y) + \frac{\partial g_1}{\partial x} (x,y) (x \epsilon - \epsilon \ind(-z \le y \le z)) \epsilon \\
&\qquad \qquad \qquad \qquad \qquad \qquad \quad + \frac{\partial g_1}{\partial y}(x,y) (\frac{\epsilon}{2} \, y - w \sqrt \epsilon) \\
&\qquad \qquad \qquad \qquad \qquad \qquad \quad + \frac{1}{2} \frac{\partial^2 g_1}{\partial y^2} (x,y) w^2 \epsilon + o(\epsilon) \Big]\dee{w} + o(\epsilon)\\
&= g_1(x,y) + \tfrac32 \epsilon g_1(x,y) + \frac{\partial g_1}{\partial x} (x,y) (x\epsilon - \epsilon \ind(-z \le y \le z)) \\ & \qquad + \frac{\partial g_1}{\partial y} (x,y) \frac\epsilon2 y  + \frac12 \frac{\partial^2 g_1}{\partial y^2} (x,y) \epsilon + o(\epsilon).
\end{align*}
Hence, $g_1$ should satisfy the PDE 
\begin{equation} \label{pde}
    0 = \frac32 g_1(x,y) + \frac{\partial g_1}{\partial x} (x,y) (x - \ind(-z \le y \le z)) + \frac{\partial g_1}{\partial y} (x,y) \frac{y}{2} + \frac12 \frac{\partial^2 g_1}{\partial y^2} (x,y),
\end{equation}
for $0 \le x \le 1$, $y \in \R$, subject to $g_1 \ge 0$ and
\begin{equation}
    \label{3.2}
    \int_0^1 \int_{\R} g_1(x,y) \dee y \dee x = 1.
\end{equation}
Note that as a consequence of our use of Proposition \ref{prop2} in the above derivation, the validity of the PDE \eqref{pde} is predicated on the initial conditions $Z(0) = 0$, $B(0) = 0$.

We now establish that the distribution of $(Z(1), B(1))$ is identical to the stationary distribution of a Markov process $((\chi(t), \beta(t)) : t \ge 0)$ having \textit{stationary} transition probabilities. Note that Proposition \ref{prop2} suggests considering the process $((t^{-1} Z(t), t^{1/2} B(t)) : t \ge 0)$. This latter process, while Markov, clearly exhibits non-stationary transition probabilities.

The connection to a process having stationary transition structure arises as a consequence of a deterministic exponential time change. Specifically, put 
\begin{equation*}
    \beta(t) = \frac{B(e^t)}{e^{t/2}}
\end{equation*}
and 
\begin{equation*}
    \chi(t) = e^{-t} Z(e^t)
\end{equation*}
for $t \in \R$. 

In stating the next result, we use the notation $\eqd$ to denote equality in distribution. 

\begin{thm}~ \label{thm 2}
    \begin{enumerate}[label=\alph*)]
        \item \label{thm 2 a}
    The process $((\chi(t), \beta(t)) : -\infty < t < \infty)$ is a stationary solution of the two-dimensional stochastic differential equation (SDE) 
    \begin{equation} \label{7}
        \begin{split}
            \dee \chi(t) &= -\chi(t) \dee t + \ind(-z \le \beta(t) \le z) \dee t, \\
            \dee \beta(t) &= -\frac{\beta(t)}{2} + \dee \widetilde B(t),
        \end{split}
    \end{equation}
    where $(\widetilde B(t) : t \ge 0)$ is a standard Brownian motion. 
    \item \label{thm 2 b} $Z(1) \eqd \int_0^\infty e^{-u} \ind(-z \le \beta(u) \le z) \dee u$, where $(\beta(t) : -\infty < t < \infty)$ is a stationary Ornstein-Uhlenbeck process satisfying \eqref{7}.
    \end{enumerate}
\end{thm}
\begin{proof}
    We observe that $(e^{-t/2} B(e^t) : -\infty < t < \infty)$ is a Gaussian process with mean 0 and covariance function given by 
    \begin{align*}
        c(t, t+u) &= \cov(e^{-t/2} B(e^t), e^{-(t+u)/2} B(e^{t+u})) \\
        &= e^{-t} e^{-u/2} \cov(B(e^t), B(e^{t+u})) \\
        &= e^{-u/2},
    \end{align*}
    which is the covariance function of a stationary Ornstein-Uhlenbeck process satisfying the stated SDE for $(\beta(t): t \ge 0)$; see for example p.~137 of \cite{arnold1974stochastic}. In addition,
    \begin{align}
        \chi(t) &= e^{-t} \int_0^{e^t} \ind(-z \le \frac{B(s)}{\sqrt s} \le z) \dee s \nonumber \\
        &= e^{-t} \int_{-\infty}^t \ind(-z \le \frac{B(e^u)}{e^{u/2}} \le z) e^u \dee u \label{8}
    \end{align}
    so 
    \begin{align*}
        \dee \chi(t) &= -e^{-t} \left( \int_{-\infty}^t \ind(-z \le \frac{B(e^u)}{e^{u/2}} \le z) e^u \dee u \right)\dee t + \ind(-z \le \frac{B(e^t)}{e^{t/2}} \le z) \dee t \\ 
        &= -\chi(t) \dee t + \ind(-z \le \beta(t) \le z) \dee t.
    \end{align*}
    For part b), we note that \eqref{8} implies that 
    \begin{align*}
        \chi(1) &= \int_{-\infty}^1 e^{-(1-u)} \ind(-z \le \beta(u) \le z) \dee u \\
        &= \int_1^\infty e^{-(u-1)} \ind(-z \le \beta(1-u) \le z) \dee u \\
        &\eqd \int_0^\infty e^{-u} \ind(-z \le \beta(u) \le z) \dee u,
    \end{align*}
    where the equality in distribution follows from the standard fact that a stationary Ornstein-Uhlenbeck process is reversible; hence $(\beta(t) : t \ge 1) \eqd (\beta(1-t) : 0 \le t \le 1)$; see pp.~104--116 of \cite{pavliotis2014stochastic}. Finally, note that $\chi(1) \eqd Z(1)$ for $t \in \R$. 
\end{proof}

Part \ref{thm 2 b} of Theorem \ref{thm 2} provides an alternative representation of $Z(1)$ in terms of an Ornstein-Uhlenbeck process. We use part \ref{thm 2 a} as a vehicle for constructing a numerical procedure for computing $g_1$ in Section \ref{sec:4}.

\begin{remark}
    In \cite{breiman1967first}, the first exit time of Brownian motion from a square-root boundary was investigated via the same transformation $e^{-t/2} B(e^t)$. Theorem \ref{thm 2} shows that this transformation is also useful to derive information about the occupation time of Brownian motion within a square-root boundary. 
\end{remark}

Our final result is a two-dimensional PDE that provides the entire MGF of $Z(1)$ when solved (as opposed to part \ref{prop1:b} of Proposition \ref{prop1}, which yields the MGF only at a single $\theta$).

\begin{prop}
    Suppose that $u: \R_+ \times \R$ is a bounded, twice continuously differentiable function that satisfies the PDE 
    \begin{equation*}
        -x \ind(-z \le y \le z) u(x,y) - x \frac{\partial u}{\partial x} (x,y) - \frac{y}{2} \frac{\partial u}{\partial y} (x,y) + \frac{1}{2} \frac{\partial^2 u}{\partial y^2} (x,y) = 0,
    \end{equation*}
    for $x \ge 0$, $y \in \R$, subject to $u(0,y) \equiv1 $, for $y \in \R$. Then, 
    \begin{equation*}
        u(\theta, y) = \E [ \exp(-\theta \int_0^\infty e^{-u} \ind(-z \le \chi(u) \le z) \dee u) \mid X(0) = y],
    \end{equation*}
    for $\theta \ge 0$, $y \in \R$.
\end{prop}

\begin{proof}
    Note that It\^o's formula ensures that 
    \begin{align*}
    & \dee (\exp(-\theta \int_0^t e^{-u} \ind(-z \le \chi(s) \le z) \dee s) u(\theta e^{-t}, \chi(t))) \\
    &= \exp\left(-\theta \int_0^t e^{-u} \ind(-z \le \chi(s) \le z) \dee s\right)  \\
    &\quad \cdot \Bigg\{ -\theta e^{-t} \ind(-z \le \chi(t) \le z) u(\theta e^{-t}, \chi(t)) \dee t \\
    &\quad~ - \theta e^{-t} \frac{\partial}{\partial x} u(\theta e^{-t}, \chi(t)) \dee t - \frac{\partial}{\partial y} u(\theta e^{-t}, \chi(t)) \dee t \Bigg\} \\
    &= \exp\left(-\theta \int_0^t e^{-u} \ind(-z \le \chi(s) \le z) \dee s\right) \frac{\partial}{\partial y} u(\theta e^{-t}, \chi(t)) \dee B(t).
    \end{align*}
    Because $u$ is bounded and satisfies the PDE, it follows that 
    \begin{equation*}
        M'(t) = \exp(-\theta \int_0^t e^{-u} \ind(-z \le \chi(s) \le z) \dee s) u(\theta e^{-t}, \chi(t))
    \end{equation*}
    is a martingale. Hence, 
    \begin{equation*}
        \E [ \exp(-\theta \int_0^t e^{-u} \ind(-z \le \chi(s) \le z) \dee s) u(\theta e^{-t}, \chi(t)) \mid \chi(0) = x] = u(\theta, x)
    \end{equation*}
    for $\theta \ge 0$ and $x \in \R$. Sending $t \to \infty$ and applying the bounded convergence theorem yields the result.
\end{proof}

To obtain the MGF of $Z(1)$, integrate $u(\theta, \cdot)$ against the standard normal density $\phi$.

\section[A Numerical Scheme for Solving the PDE for the Density of Z(1)]{A Numerical Scheme for Solving the PDE for the Density of $Z(1)$} \label{sec:4}
The PDE \eqref{pde} is difficult to discretize numerically because diffusion occurs in only one of the two dimensions. Consequently, it is difficult to obtain stable numerical solutions to \eqref{pde} when implementing conventional discretization schemes for numerical PDE solving. However, we do obtain stable numerical solutions that agree with Monte Carlo simulations of the density when we discretize using an appropriate approximating Markov chain. This is the approach followed by Kushner \cite{Kushner2001} in his development of numerical schemes for such singular diffusions. We now give a quick outline of the approach.

Note that if $f$ is a bounded, twice continuously differentiable function with bounded derivatives, then It\^o's formula ensures that 
\begin{equation} \label{9}
    f(\chi(t), \beta(t)) - \int_0^t (\mathcal L f)(\chi(s), \beta(s)) \dee s
\end{equation}
is a martingale adapted to the filtration of $((\chi(s), \beta(s)) : s \ge 0)$, where 
\begin{equation} \label{10}
    (\mathcal L f)(x,y) = -(x - \ind(-z \le y \le z)) \frac{\partial f}{\partial x} (x,y) - \frac{y}{2} \frac{\partial f}{\partial y} (x,y) + \frac{1}{2} \frac{\partial^2 f}{\partial y^2} (x,y).
\end{equation}
We now construct an approximating Markov chain $((\chi_i(\epsilon), \beta_i(\epsilon)) : i \ge 0)$ on an appropriate lattice, so that 
\begin{equation*}
    \E [f(\chi(\epsilon), \beta(\epsilon)) \mid \chi(0) = x, \beta(0) = y] \approx \E[f(\chi_1(\epsilon), \beta_1(\epsilon)) \mid \chi_0(\epsilon) = x, \beta_0(\epsilon) = y]
\end{equation*}
for $\epsilon$ small. The martingale property \eqref{9} suggests the requirement that 
\begin{equation} \label{11}
    \E[ f(\chi_1(\epsilon), \beta_1(\epsilon)) \mid \chi_0(\epsilon) = x, \beta_0(\epsilon) = y ] \approx f(x,y) + \epsilon (\mathcal L f)(x,y).
\end{equation}
Suppose that we use a lattice $\{(j \gridx, k \gamma) : 0 \le j \le j_0, |k| \le k_0\}$ as the state space for our approximating Markov chain. Then, \eqref{10} and \eqref{11} require that the transition probabilities of $((\chi_1(\epsilon), \beta_i(\epsilon)) : i \ge 0)$ are chosen so that 
\begin{equation} \label{12}
    \begin{split}
        &\E[f(\chi_1(\epsilon), \beta_1(\epsilon)) \mid \chi_0(\epsilon) = j\gridx, \beta_0(\epsilon) = k\gamma] \\
        &=f(j\gridx, k\gamma) - \left[ \frac{f(j\gridx, k\gamma) - f((j-1) \gridx, k\gamma)}{\gridx} \right] j\gridx \epsilon \\
        &\qquad - \left[ \frac{f(j\gridx, (k+1) \gamma) - f(j\gridx, (k-1) \gamma)}{2\gamma} \right] \frac{k \gamma}{2} \epsilon \\
        &\qquad + \left[ \frac{f(j\gridx, (k+1) \gamma) - 2 f(j\gridx, k \gamma) + f(j\gridx, (k-1) \gamma)}{\gamma^2} \right] \frac{1}{2} \epsilon
    \end{split}
\end{equation}
for $|k \gamma| > z$. A key aspect of this construction is the use of a backward difference to approximate $\frac{\partial f}{\partial x}(x,y)$, rather than a forward or central difference. This choice avoids introducing negative coefficients to the function values at points other than $(j\gridx, k \gamma)$, as this would be inconsistent with a Markov chain approximation. Similarly, when $|k \gamma| \le z$, we require that 
\begin{equation} \label{13}
    \begin{split}
        &\E[f(\chi_1(\epsilon), \beta_1(\epsilon)) \mid \chi_0(\epsilon) = j\gridx, \beta_0(\epsilon) = k\gamma] \\
        &=f(j\gridx, k\gamma) + \left[ \frac{f((j+1)\gridx, k\gamma) - f(j \gridx, k\gamma)}{\gridx} \right] (1-j\gridx) \epsilon \\
        &\qquad - \left[ \frac{f(j\gridx, (k+1) \gamma) - f(j\gridx, (k-1) \gamma)}{2\gamma} \right] \frac{k \gamma}{2} \epsilon \\
        &\qquad + \left[ \frac{f(j\gridx, (k+1) \gamma) - 2 f(j\gridx, k \gamma) + f(j\gridx, (k-1) \gamma)}{\gamma^2} \right] \frac{1}{2} \epsilon,
    \end{split}
\end{equation}
wherein a forward difference, rather than a backward difference, is employed to preserve the non-negativity of the coefficients in the Markov chain approximation.

We now re-arrange \eqref{12} so that the right-hand side resembles an expectation over the points $(j\gridx, k \gamma)$, $((j-1) \gridx, k\gamma)$, $(j\gridx, (k+1)\gamma)$, and $(j\gridx, (k-1)\gamma)$:

\begin{equation*}
    \begin{split}
        &\E[f(\chi_1(\epsilon), \beta_1(\epsilon)) \mid \chi_0(\epsilon) = j \gridx, \beta_1(\epsilon) = k \gamma] \\
        &= (1 - j\epsilon - \epsilon/\gamma^2) f(j\gridx, k\gamma) + j\epsilon f((j-1)\gridx, k\gamma) \\
        &\quad+ (\epsilon/(2\gamma^2) - k\epsilon/4) f(j\gridx, (k+1)\gamma) + (\epsilon/(2\gamma^2) + k\epsilon/4) f(j\gridx, (k-1) \gamma).
    \end{split}
\end{equation*}
Note that the weights on the function values sum to 1, as required by a probability. To make the weights non-negative, we need to choose $\gridx$ and $\gamma$ so that $1 - \epsilon/\gridx - \epsilon/\gamma^2 \ge 0$ (since $j\gridx \in [0,1]$). The choice $\epsilon/\gridx = \epsilon/\gamma^2 = \tfrac12$ satisfies this requirement; equivalently, $\gridx = 2\epsilon$ and $\gamma = \sqrt{2 \epsilon}$. Because $j_0 \gridx = 1$ is the right endpoint of $[0,1]$, it is the case that $j_0$ and $\epsilon$ are related via $\epsilon = 1/(2j_0)$; hence $\gamma = \sqrt{1/j_0}$. To obtain a reasonable value for $k_0$, we approximate $P(|\mathcal N(0,1)| \ge 6) \approx 0$, so we set $k_0$ such that $\lfloor k_0 j_0^{-1/2} \rfloor = 6$. Hence, for $|k j_0^{-1/2}| > z$ and $|k| < k_0$, we use as our approximation a Markov chain with transition probabilities 
\begin{equation} \label{14}
    \begin{split}
    &P_\epsilon((j\gridx, k\gamma), (j\gridx, k\gamma)) = 1/2 - j\gridx / 2, \\
    &P_\epsilon((j\gridx, k\gamma), ((j-1)\gridx, k\gamma)) = j\gridx / 2, \\
    &P_\epsilon((j\gridx, k\gamma), (j\gridx, (k+1)\gamma)) = 1/4 - k\gridx / 8,\\
    &P_\epsilon((j\gridx, k\gamma), (j\gridx, (k-1)\gamma)) = 1/4 + k\gridx / 8.
    \end{split}
\end{equation}
When $|k| = k_0$, we move the mass associated with transitioning to $(j\gridx, (k_0 + 1) \gamma)$ or $(j\gridx, (-k_0-1)\gamma)$ to the self-transition taking the chain back to state $(j\gridx, k_0 \gamma)$ or $(j\gridx, -k_0\gamma)$, respectively. 

For $|kj_0^{-1/2}| \le z$, an analogous analysis leads to the transition probabilities 
\begin{equation} \label{15}
    \begin{split}
    &P_\epsilon((j\gridx, k\gamma), (j\gridx, k\gamma)) = 1/2 - (1-j\gridx) / 2, \\
    &P_\epsilon((j\gridx, k\gamma), ((j-1)\gridx, k\gamma)) = (1-j\gridx) / 2, \\
    &P_\epsilon((j\gridx, k\gamma), (j\gridx, (k+1)\gamma)) = 1/4 - k\gridx / 8,\\
    &P_\epsilon((j\gridx, k\gamma), (j\gridx, (k-1)\gamma)) = 1/4 + k\gridx / 8,
    \end{split}
\end{equation}
when $|k| < k_0$, with the same choice made at the boundary $|k| = k_0$ as was made for the case $|k j_0^{-1/2}| > z$ above. The relations \eqref{14} and \eqref{15} fully define the approximating Markov chain we use in our numerics. Note that the state space of this Markov chain has approximately $6j_0^{3/2}$ states.

\begin{remark}
    Note that when $|kj_0^{-1/2}| \le z$, the approximating Markov chain can only make moves in the positive $x$ direction, while if $|kj_0^{-1/2}| > z$, the chain can only make moves in the negative $x$-direction, as expected from the SDE for $\chi$ in Theorem \ref{thm 2}.
\end{remark}
\begin{remark}
    An interesting feature of this discretization, in addition to the particular choice of forward / backward differencing to compute $\frac{\partial f}{\partial x} (x,y)$, is the use of a much finer grid in the $x$-coordinate than in the $y$-coordinate. This is a consequence of the fact that diffusion occurs only in the $y$-direction, so a coarse grid can only be used in the $y$ coordinate.
\end{remark}

Having defined the Markov transition matrix by 
$$P_\epsilon = (P_\epsilon((j/j_0, kj_0^{-1/2}), (j'/j_0, k'j_0^{-1/2})) : 0 \le j, j' \le j_0, |k|, |k'| \le \lceil 6 j_0^{1/2} \rceil ),$$
we can compute the associated stationary distribution $\pi_\epsilon = (\pi_\epsilon(j/j_0, k j_0^{-1/2}) : 0 \le j \le j_0, |k| \le \lceil 6j_0^{1/2} \rceil)$ by solving 
\begin{equation} \label{16}
    \pi_\epsilon = \pi_\epsilon P_\epsilon
\end{equation}
subject to 
\begin{equation} \label{17}
    \sum_{j=0}^{j_0} \sum_{|k| \le \lceil 6j_0^{1/2} \rceil} \pi_\epsilon (j/j_0, kj_0^{-1/2}) = 1. 
\end{equation}
In Figure \ref{grid}, we use \eqref{16} and \eqref{17} to compute the density of $Z(1)$ at various values of $1-\level$ by solving the appropriate linear system directly via sparse LU factorization. These computations were performed with $j_0 = 10^4$, and the code can be found at \url{https://github.com/montefischer/ci_coverage}. Portions of the code were developed with the assistance of Google Gemini Pro (\url{https://gemini.google.com}) used between November 2025 and January 2026, and Claude Opus 5.5 used in September and October 2026 to support implementation, debugging, and plotting. All code was reviewed and validated by the authors. 

\begin{remark}
    When $((\chi(t), \beta(t)) : t \ge 0)$ is initialized with its stationary distribution $\pi$, the martingale property \eqref{9} ensures that 
    \begin{equation*}
        \int_0^1 \int_{\R} \pi(\rd x, \rd y) (\mathcal L f) (x,y) = 0
    \end{equation*}
    for all suitably smooth and bounded functions $f$. This is the \textit{basic adjoint relationship} (BAR) that characterizes $\pi$. Assuming that $\pi$ has a smooth density $(p(x,y) : x \in [0,1], y \in \R)$, a formal integration-by-parts argument (see, for example, p.~220 of \cite{karlin1981second}) shows that $p$ satisfies the PDE \eqref{pde}. 
\end{remark}
\begin{remark}
    Starting from \eqref{16}, it is straightforward to formally argue that as $\epsilon \downarrow 0$, $\pi_\epsilon$ should converge to a solution $p$ of \eqref{pde}. This is as expected, given that $\pi_\epsilon$ approximates the two-dimensional stationary density $p$, which must coincide with $g_1$.
\end{remark}

The family of densities corresponding to different choices of $\level$ exhibits a remarkable variety of behaviors, as Figure \ref{grid} shows. An observation of particular interest is the right-skewness of the density of $Z(1)$ for $\level > 0.5$. A statistical interpretation of this fact is that a typical sequence of conventional confidence intervals calibrated with typical choices of $(1-\level)$ will cover the true parameter for a larger than $(1-\level)$ proportion of the time. 

\begin{figure} 
    \includegraphics[width=\linewidth, alt={Nine panels showing the density of Z(1) on the interval from 0 to 1, one for each confidence level 1 minus delta from 0.10 to 0.95. At low levels the mass piles up near 0; near level 0.5 there is a spike at 0 and a broad hump near 0.8; at high levels the mass piles up near 1, so the median exceeds the mean, for example a median of 0.99 against a mean of 0.90 at level 0.90.}]{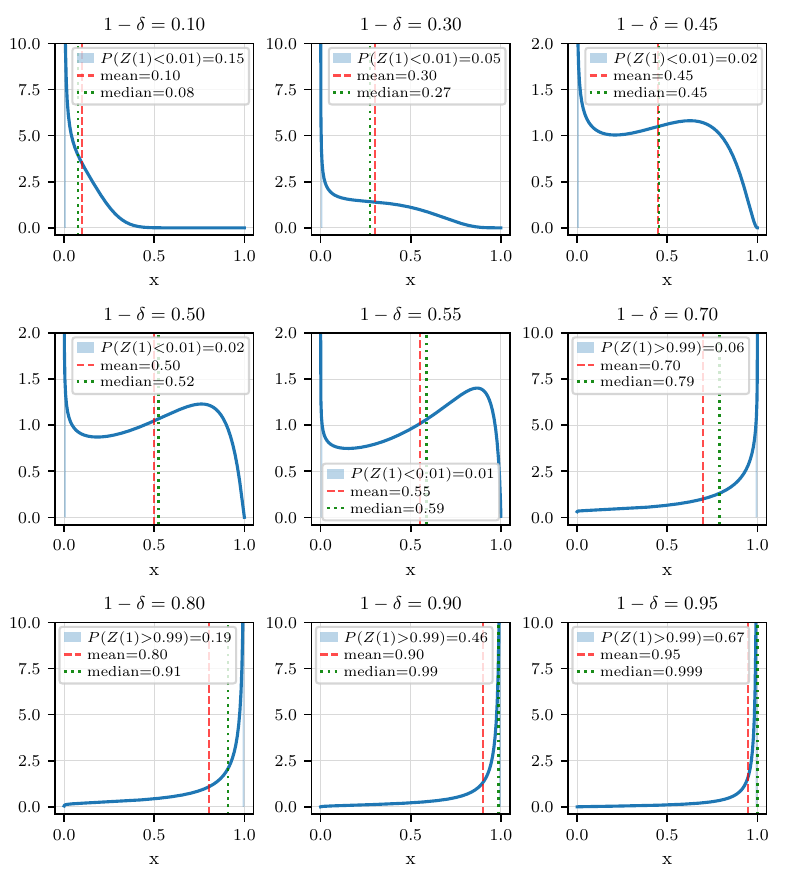}
    \caption{Density of $Z(1)$ for various values of $\level$ computed as the stationary distribution of the Markov chain associated with the PDE \eqref{pde}. Vertical axes are clipped at 2 or 10 to better exhibit the characteristic features of each density. Because densities exhibit asymptotic behavior near either $x=0$ or $x=1$, for $\level$ near $1$ or $0$, estimates of the quantities $P(Z(1) < 0.01)$ or $P(Z(1) > 0.99)$ as appropriate are included.}
    \label{grid}
\end{figure}

A numerical solution of the PDE \eqref{pde} for the joint density $g_1(x,y)$ of $(Z(1), B(1))$ also yields numerical solutions for the conditional density $g_1(y \mid x)$ of Brownian motion $B(1)$ given $Z(1)=x$. 
For $t \ne 1$, the conditional density of $B(t)$ is just a scaled version of that for $B(1)$ due to Proposition \ref{prop2}. 
We plot these conditional densities for various values of $\level$ and values of $Z(1)$ in Figure \ref{fig conditional densities}. The distributions are smooth and symmetric, and display an interesting variety of behaviors. When the coverage proportion is small relative to $\delta$ (as shown by $x = 0.1$ in Figure \ref{fig conditional densities}), the conditional distribution of $B(1)$ has two modes, each located a moderate distance outside of $[-z, z]$, as would be expected for Brownian paths that lie outside $[-z, z]$ with high probability. When the coverage proportion is high (as shown by $x = 0.9$ in Figure \ref{fig conditional densities}), the conditional density is largely concentrated within the region $[-z, z]$, and becomes bimodal as $\delta \downarrow 0$. 

\begin{figure}
    \includegraphics[width=\linewidth, alt={Five stacked panels, for confidence levels 0.20, 0.50, 0.80, 0.90 and 0.95, each showing the conditional density of B(1) given Z(1) equal to 0.1, 0.5 and 0.9, with vertical lines at plus and minus z. Given 0.1 the density has two modes just outside plus and minus z; given 0.9 it is concentrated between minus z and z, flattening as the level rises; given 0.5 it is intermediate.}]{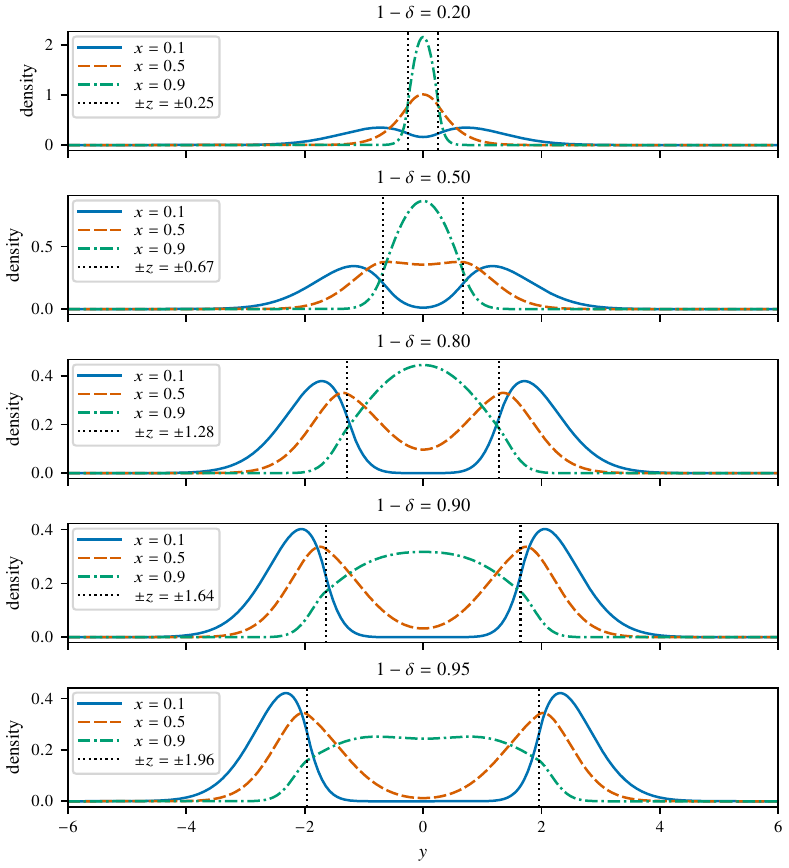}
    \caption{Conditional density $g_1(\cdot \mid x)$ of $B(1)$ given a specified occupation time $Z(1)=x$ of Brownian motion in the region with square root boundary associated with the confidence level $1-\level$. Conditional densities plotted for several choices of $\level$, with $x \in \{0.1, 0.5, 0.9\}$.}
    \label{fig conditional densities}
\end{figure}

Monte Carlo simulation supplements the PDE approach by allowing the study of paths of $B$ as a function of $Z(1)$.  We discretize $(t_i : 0 \le i \le n)$ such that $0 = t_0 < t_1 < \cdots < t_{n-1} < t_n = 1$, and simulate $(B(t_i) : 0 \le i \le n)$ as the cumulative sum of independent normal random variables with mean 0 and variances $(t_i - t_{i-1})$, $i = 1, \dots, n$. Then, we let 
\[Z_n(1) = \sum_{i=1}^n \ind(-v \le \frac{B(t_i)}{\sqrt{t_i}} \le v ) (t_i - t_{t-1}). \]
As the mesh size of the discretization decreases to zero, that is, $\max\{t_{i+1} - t_i : 1 \le i \le n-1\} \downarrow 0$, it holds that $Z_n(1)$ converges in distribution to $Z(1)$. 

Taking $m =10^5$ i.i.d.~samples of $Z_n(1)$ for various values of $\level \in (0, 1)$, with $n=10^5$ evenly spaced time-discretization points each, we tabulate several quantities of interest regarding $Z(1)$ in Table \ref{tab:statistics}. These simulation estimates confirm the interesting property of $Z(1)$ that for $\level < \tfrac12$, a significant majority of paths exceed the expected coverage proportion $\E Z(1)$. For example, in expectation, the proportion of time that a sequence of 90\% confidence intervals ($\level = 0.1$) covers the true parameter is, of course, 90\%, but with probability roughly $\tfrac34$, the time-averaged coverage will exceed 90\%, and with probability $\tfrac12$, the time-averaged coverage of the sequence will exceed 98.5\%. This phenomenon is reversed for confidence intervals calibrated to have low coverage (as $\level \to 1$), which is an atypical choice in the statistics context. 

\begin{table}[h]
\centering
\begin{tabular}{l|rrrr}
\toprule
$\bm{\level}$ & $\E Z(1)$ & $\text{median}( Z(1))$ & $\sqrt{\var Z(1)}$ & $\p(Z(1) > 1-\level)$ \\
\midrule
0.01 & $0.9900 \pm 0.0003$ & $0.9999 \pm 0.0001$ & $0.0528 \pm 0.0016$ & $0.9136 \pm 0.0017$ \\
0.05 & $0.9502 \pm 0.0008$ & $0.9991 \pm 0.0001$ & $0.1262 \pm 0.0016$ & $0.8078 \pm 0.0024$ \\
0.10 & $0.9001 \pm 0.0011$ & $0.9860 \pm 0.0003$ & $0.1796 \pm 0.0016$ & $0.7401 \pm 0.0027$ \\
0.30 & $0.6997 \pm 0.0017$ & $0.7901 \pm 0.0023$ & $0.2747 \pm 0.0011$ & $0.6021 \pm 0.0030$ \\
0.45 & $0.5497 \pm 0.0018$ & $0.5894 \pm 0.0029$ & $0.2837 \pm 0.0008$ & $0.5401 \pm 0.0031$ \\
0.50 & $0.4998 \pm 0.0017$ & $0.5224 \pm 0.0029$ & $0.2780 \pm 0.0008$ & $0.5230 \pm 0.0031$ \\
0.55 & $0.4499 \pm 0.0017$ & $0.4567 \pm 0.0028$ & $0.2683 \pm 0.0008$ & $0.5075 \pm 0.0031$ \\
0.70 & $0.2997 \pm 0.0013$ & $0.2744 \pm 0.0022$ & $0.2157 \pm 0.0007$ & $0.4637 \pm 0.0031$ \\
0.90 & $0.0998 \pm 0.0006$ & $0.0766 \pm 0.0008$ & $0.0897 \pm 0.0004$ & $0.4143 \pm 0.0031$ \\
0.95 & $0.0499 \pm 0.0003$ & $0.0365 \pm 0.0004$ & $0.0473 \pm 0.0003$ & $0.4013 \pm 0.0030$ \\
0.99 & $0.0100 \pm 0.0001$ & $0.0070 \pm 0.0001$ & $0.0099 \pm 0.0001$ & $0.3914 \pm 0.0030$ \\

\bottomrule
\end{tabular}
\caption{Monte Carlo estimates of statistics for the limiting coverage RV $Z(1)$, for $(1-\level)$ confidence intervals. Statistics estimated using $10^5$ i.i.d.~Brownian paths, simulated with step size $10^{-5}$. For the estimate of the median when $\delta = 0.01$, $10^4$ paths with step size $10^{-6}$ were used instead. We report 95\% confidence intervals for the corresponding moments.}
\label{tab:statistics}
\end{table}

We focus our attention on the $1-\level = 0.95$ case, corresponding to the last plot in Figure \ref{fig conditional densities}. For each of the values $x \in \{0.1, 0.5, 0.9\}$ in this plot, we simulate Brownian sample paths conditional on $Z(1) \in [x-0.01, x]$. The resulting conditional paths are displayed in Figure \ref{fig conditional paths}. These simulated paths reveal a variety of interesting qualitative behaviors. For example, Brownian paths conditional on low occupation times of about 0.1 lie well outside $[-z \sqrt t, z \sqrt t]$ for most of the interval. The visits to $[-z \sqrt t, z \sqrt t]$ tend to occur early, near $t=0$, as evidenced by the sample paths conditional on $Z(1) \in [0.09, 0.1]$ in Figure \ref{fig conditional paths}. This behavior is not as strong for Brownian paths conditional on `moderate' occupation times $(Z(1) \in [0.49, 0.5]$); rather, these paths appear to cross in and out of the square root boundary region across the entire interval $[0,1]$. For paths conditioned to have `high' occupation times ($Z(1) \in [0.945, 0.955]$), excursions outside the square root boundary tend to concentrate near $t=0$. In each case, the shape of the conditional density of $B(1) \mid Z(1)$ from Figure \ref{fig conditional densities} is approximated by the terminal states occupied by the conditional paths, providing a simulation-based confirmation of our numerical solution for the joint density $g_1$. 

\begin{figure}
    \centering
    \includegraphics[width=0.75\linewidth, alt={100 simulated Brownian paths on the time interval from 0 to 1 with occupation Z(1) between 0.09 and 0.10. After a short time near 0 the paths split into an upper and a lower bundle that stay outside the dashed boundary plus or minus 1.96 times the square root of t.}]{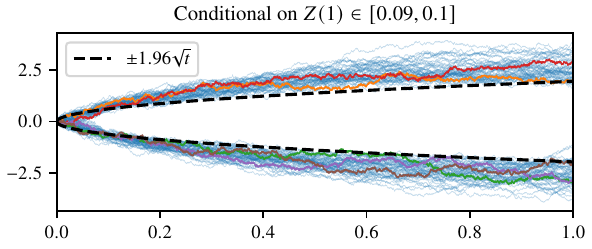}
    \includegraphics[width=0.75\linewidth, alt={100 simulated Brownian paths with occupation Z(1) between 0.49 and 0.50. The paths split into upper and lower bundles that track the dashed boundary plus or minus 1.96 times the square root of t, crossing in and out of it throughout the interval.}]{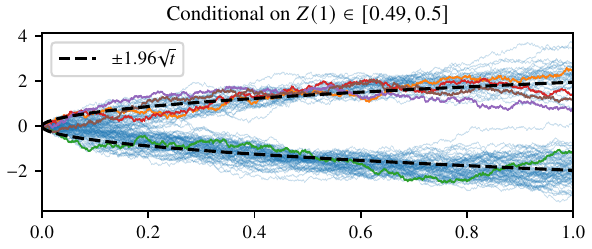}
    \includegraphics[width=0.75\linewidth, alt={100 simulated Brownian paths with occupation Z(1) between 0.945 and 0.955. The paths stay mostly inside the dashed boundary plus or minus 1.96 times the square root of t, with brief excursions outside it.}]{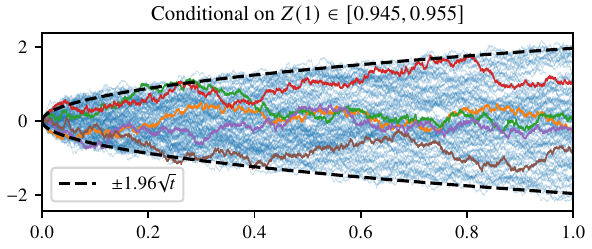}
    \includegraphics[width=0.75\linewidth, alt={100 simulated Brownian paths with occupation Z(1) between 0.99 and 1. The paths stay almost entirely inside the dashed boundary plus or minus 1.96 times the square root of t for the whole interval.}]{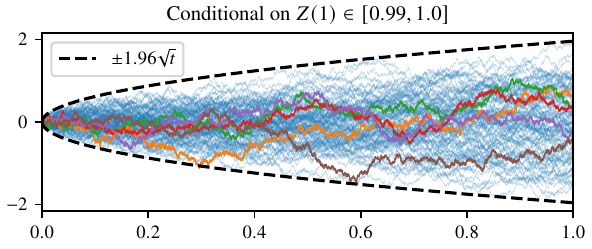}
    \caption{Paths of Brownian motion conditional on having occupation $Z(1)$ within $[0.09, 0.1]$, $[0.49, 0.50]$, $[0.945, 0.955]$, and $[0.99, 1]$ (top to bottom), for $1-\level=0.95$. A total of 100 conditional paths per plot are shown, with 5 randomly selected paths highlighted. A discretization time step of $10^{-4}$ was used to simulate i.i.d.~copies of Brownian motion, and conditional samples were produced by simple acceptance-rejection.}
    \label{fig conditional paths}
\end{figure}

\section{Extension to Confidence Regions} \label{sec multivariate}

We now briefly discuss the extension from the setting of confidence intervals to confidence regions. Suppose, in particular, that $\alpha_d(t) \in \R^d$ (interpreted as a column vector) is an estimator for the $d$-dimensional parameter $\alpha_d$ satisfying:
\begin{assumption} \label{a3}
    There exists a positive definite symmetric matrix $C$ such that 
    \begin{equation*}
        (\epsilon^{-1/2} (\alpha_d(t/\epsilon) - \alpha_d) : t > 0) \Rightarrow (C^{1/2} B_d(t) : t > 0)
    \end{equation*}
    as $\epsilon \downarrow 0$ in $D_d(0, \infty)$, where $B_d = (B_d(t) : t \ge 0)$ is a $d$-dimensional standard Brownian motion and $D_d(0, \infty)$ is the $d$-dimensional version of $D(0, \infty)$. 
\end{assumption}
\begin{assumption} \label{a4}
    There exists an estimator $v_d(t)$ such that 
    \begin{equation*}
        v_d(t) \to C \quad \as
    \end{equation*}
    as $t \to \infty$.
\end{assumption}
For $\delta \in (0,1)$, choose $z$ so that $\p(B_d(1)^\top B_d(1) \le z^2) = 1-\delta$, and put 
\begin{equation*}
    \mathscr P(t) = \frac{1}{t} \int_0^t \ind((\alpha_d(s) - \alpha)^\top v_d(s) (\alpha_d(s) - \alpha) \le z^2) \dee s.
\end{equation*}
\begin{thm}
    Under \ref{a3} and \ref{a4}, 
    \begin{equation*}
        \mathscr P_d(t) \Rightarrow Z_d(1)
    \end{equation*}
    as $t \to \infty$, where 
    \begin{equation*}
        Z_d(t) = \int_0^t \ind(B_d(s)^\top B_d(s) \le z^2) \dee s.
    \end{equation*}
    Furthermore,
    \begin{equation*}
        Z_d(1) \eqd \int_0^\infty e^{-u} \ind(\chi_d(u)^\top \chi_d(u) \le z^2) \dee u,
    \end{equation*}
    where $\chi_d = (\chi_d(u) : u \ge 0)$ is a $d$-dimensional Ornstein-Uhlenbeck process satisfying the SDE 
    \begin{equation*}
        \dee \chi_d(t) = -\frac{\chi_d(t)}{2} \dee t + \dee \widetilde B_d(t),
    \end{equation*}
    where $(\widetilde B_d(t) : t \ge 0)$ is a $d$-dimensional standard Brownian motion.
\end{thm}
We omit the proof, which is virtually identical to the scalar case.

\ack The information, views, and opinions expressed herein are solely the authors' and do not necessarily represent the views of Monte Fischer's employer or its affiliates. His employer and its affiliates are not responsible for and did not verify the accuracy of the information contained herein.

\fund
There are no funding bodies to thank relating to the creation of this article.

\competing
There were no competing interests to declare which arose during the preparation or publication process of this article.

\data
Code to produce the data tabulated and plotted in this article can be found at \url{https://github.com/montefischer/ci_coverage}.
Claude Opus 5.5 (Anthropic) was used to assist literature review, to draft accessibility alt text, and to translate notebook code for Figures \ref{fig conditional densities} and \ref{fig conditional paths} into reproducible scripts with size and style matching the journal's layout; simulations, text, and proofs are our own.

\bibliographystyle{APT}
\bibliography{confidence}

\end{document}